\documentclass{amsart}
\usepackage[all]{xy}

\newtheorem{theorem}{Theorem}[section]
\newtheorem{proposition}[theorem]{Proposition}
\newtheorem{corollary}[theorem]{Corollary}

\newtheorem{claim}[theorem]{Claim}
\newtheorem{problem}[theorem]{Problem}
\newtheorem{example}[theorem]{Example}
\theoremstyle{definition}
\newtheorem{definition}[theorem]{Definition}
\newtheorem{remark}[theorem]{Remark}

\newcommand{\e}{\varepsilon}
\newcommand{\w}{\omega}
\newcommand{\IN}{\mathbb N}
\newcommand{\IR}{\mathbb R}
\newcommand{\IZ}{\mathbb Z}
\newcommand{\I}{\mathcal I}
\newcommand{\J}{\mathcal J}
\newcommand{\K}{\mathcal K}

\newcommand{\IC}{\mathbb C}
\newcommand{\HN}{\mathsf{HN}}
\newcommand{\HM}{\mathsf{HM}}
\newcommand{\NF}{\mathsf{NF}}

\newcommand{\add}{\mathrm{add}}
\newcommand{\cov}{\mathrm{cov}}
\newcommand{\non}{\mathrm{non}}
\newcommand{\cof}{\mathrm{cof}}

\title{Null-finite sets in topological groups and their applications}
\author{Taras Banakh and Eliza Jab\l o\'nska}
\address{T. Banakh: Ivan Franko National University of Lviv (Ukraine) and Jan Kochanowski University in Kielce (Poland)}
\email{t.o.banakh@gmail.com}
\address{E. Jab\l o\'nska: Department of Discrete Mathematics, Rzesz\'{o}w University of Technology (Poland)}
\email{elizapie@prz.edu.pl}

\begin{document}
\begin{abstract} In the paper we introduce and study a new family of ``small'' sets which is tightly connected with two well known $\sigma$-ideals: of Haar-null sets and of Haar-meager sets. We define a subset $A$ of a  topological group $X$ to be {\em null-finite} if there exists a convergent sequence $(x_n)_{n\in\w}$ in $X$ such that for every $x\in X$ the set $\{n\in\w:x+x_n\in A\}$ is finite. We prove that each null-finite Borel set in a complete metric Abelian group is Haar-null and Haar-meager. The Borel restriction in the above result is essential as each non-discrete metric Abelian group is the union of two null-finite sets. Applying null-finite sets to the theory of functional equations and inequalities, we prove that a mid-point convex function $f:G\to\IR$ defined on an open convex subset $G$ of a metric linear space $X$ is continuous if it is upper bounded on a subset $B$ which is not null-finite and whose closure is contained in $G$. This gives an alternative short proof of a known generalization of Bernstein-Doetsch theorem (saying that a mid-point convex function $f:G\to\IR$ defined on an open convex subset $G$ of a metric linear space $X$ is continuous if it is upper bounded on a non-empty open subset $B$ of $G$). Since Borel Haar-finite sets are Haar-meager and Haar-null, we conclude that a mid-point convex function $f:G\to\IR$ defined on an open convex subset $G$ of a complete linear metric space $X$ is continuous if it is upper bounded on a Borel subset $B\subset G$ which is not Haar-null or not Haar-meager in $X$. The last result resolves an old problem in the theory of functional equations and inequalities posed by Baron and Ger in 1983.
\end{abstract}
\maketitle

\section*{Introduction}

In 1920 Steinhaus \cite{Steinhaus} proved that for any measurable sets $A,B$ of positive Haar measure in a locally compact Polish group $X$ the sum $A+B:=\{a+b:a\in A,\;b\in B\}$ has non-empty interior in $X$ and the difference $B-B:=\{a-b:a,b\in B\}$ is a neighborhood of zero in $X$.

In \cite{Ch} Christensen extended the ``difference'' part of the Steinhaus results to all Polish Abelian groups proving that for a Borel subset $B$ of a Polish Abelian group $X$ the difference $B-B$ is a neighborhood of zero if $B$ is not Haar-null. Christensen defined a Borel subset $B\subset X$ to be {\em Haar-null} if there exists a Borel $\sigma$-additive probability measure $\mu$ on $X$ such that $\mu(B+x)=0$ for all $x\in X$.

A topological version of Steinhaus theorem was obtained by Pettis \cite{Pet} and Piccard \cite{Pic} who proved that for any non-meager Borel sets $A,B$ in a Polish group $X$ the sum $A+B$ has non-empty interior and the difference $B-B$ is a neighborhood of zero.

In 2013 Darji \cite{Darji} introduced a subideal of the $\sigma$-ideal of meager sets in a Polish group, which is a topological analog to the $\sigma$-ideal of Haar-null sets. Darji defined a Borel subset $B$ of a Polish group $X$ to be {\em Haar-meager} if there exists a continuous map $f:K\to X$ from a non-empty compact metric space $K$ such that for every $x\in X$ the preimage $f^{-1}(B+x)$ is meager in $K$.
By \cite{Jab}, for every Borel subset $B\subset X$ which is not Haar-meager in a Polish Abelian group $X$ the difference $B-B$ is a neighborhood of zero.

It should be mentioned that in contrast to the ``difference'' part of the Steinhaus and Piccard--Pettis theorems, the ``additive'' part can not be generalized to non-locally compact groups: by \cite{MZ}, \cite{Jab} each non-locally compact  Polish Abelian group $X$ contains two closed  sets $A,B$ whose sum $A+B$ has empty interior but $A,B$ are neither Haar-null nor Haar-meager in $X$. In the Polish group $X=\IR^\w$, for such sets $A,B$ we can take the positive cone $\IR_+^\w$.

Steinhaus-type theorems has significant applications in the theory of functional equations and inequalities. For example, the ``additive'' part of Steinhaus Theorem can be applied to prove that a mid-point convex function $f:\IR^n\to \IR$ is continuous if it is upper bounded on some measurable set $B\subset\IR^n$ of positive Lebesgue measure (see e.g. \cite[p.210]{Kuczma}). We recall that a function $f:G\to\IR$ defined on a convex subset of a linear space is {\em mid-point convex} if $f\left(\frac{x+y}{2}\right)\leq \frac{f(x)+f(y)}{2}$ for any $x,y\in G$. Unfortunately, due to the example of Matou\u{s}kov\'{a} and Zelen\'{y} \cite{MZ} we know that the ``additive'' part of the Steinhaus theorem does not extend to non-locally compact Polish Abelian groups. This leads to the following natural problem whose ``Haar-null'' version was posed by Baron and Ger in \cite[P239]{BG}.
\smallskip

\noindent{\bf Problem} (Baron, Ger, 1983). {\em Is the continuity of an additive or mid-point convex function $f:X\to\IR$ on a Banach space $X$ equivalent to the upper boundedness of the function on some Borel subset $B\subset X$ which is not Haar-null or not Haar-meager?}
\smallskip

In this paper we give the affirmative answer to the Baron--Ger  Problem applying a new concept of a null-finite set, which will be introduced in Section~\ref{s:Haar-finite}. In Section~\ref{s:Steinhaus} we show that non-null-finite sets $A$ in  metric groups $X$ possess a Steinhaus-like property: if a subset $A$ of a metric group $X$ is not null-finite, then $A-\bar A$ is a neighborhood of $\theta$ and for some finite set $F\subset X$ the set $F+(A-A)$ is a neighborhood of zero in $X$.
In Sections~\ref{s:Haar-meager} and \ref{s:Haar-null} we shall prove that Borel null-finite sets in Polish Abelian groups are Haar-meager and Haar-null. On the other hand, in Example~\ref{ex1} we show that the product $\prod_{n\in\IN}C_n$ of cyclic groups contains a closed Haar-null set which is not null-finite and in Example~\ref{ex2} we show that the product $\prod_{n\in\IN}C_{2^n}$ contains a null-finite $G_\delta$-set which cannot be covered by countably many closed Haar-null sets. In Section~\ref{s:union} we prove that each non-discrete metric Abelian group is the union of two null-finite sets and also observe that sparse sets (in the sense of Lutsenko--Protasov \cite{LP}) are null-finite.   In Section~\ref{s:functional} we apply null-finite sets to prove that an additive functional $f:X\to \IR$ on a Polish Abelian group is continuous if it is upper bounded on some subset $B\subset X$ which is not null-finite in $X$. In Section~\ref{s:Banach} we generalize this result to additive functions with values in Banach and locally convex spaces. Finally, in Section~\ref{s:Jensen}, we prove a generalization of Bernstein--Doetsch theorem and next apply it in the proof of a continuity criterion for mid-point convex functions on complete metric linear spaces, thus answering the Baron--Ger Problem. In Section~\ref{s:prob} we pose some open problems related to null-finite sets.

\section{Preliminaries}

{\em All groups considered in this paper are Abelian.} The neutral element of a group will be denoted by $\theta$. For a group $G$ by $G^*$ we denote the set of non-zero elements in $G$. By $C_n=\{z\in\IC:z^n=1\}$ we denote the cyclic group of order $n$.

 By a ({\em complete}) {\em metric group} we understand an Abelian group $X$ endowed with a (complete) invariant metric $\|\cdot-\cdot\|$. The invariant metric $\|\cdot-\cdot\|$ determines (and can be recovered from) the {\em prenorm} $\|\cdot\|$ defined by $\|x\|:=\|x-\theta\|$. So, a metric group can be equivalently defined as a group endowed with a prenorm.

Formally, a {\em prenorm} on a group $X$ is a function $\|\cdot\|:X\to\IR_+:=[0,\infty)$ satisfying three axioms:
\begin{itemize}
\item $\|x\|=0$ iff $x=\theta$;
\item $\|-x\|=\|x\|$;
\item $\|x+y\|\le \|x\|+\|y\|$
\end{itemize}
for any $x,y\in X$; see \cite[\S3.3]{AT}.

Each metric group is a topological group with respect to the topology, generated by the metric.

All linear spaces considered in this paper are over the field $\IR$ of real numbers. By a {\em metric linear space} we understand a linear space endowed with an invariant metric.

A non-empty family $\I$ of subsets of a set $X$ is called {\em ideal of sets on $X$} if $\I$ satisfies the following conditions:
\begin{itemize}
\item $X\notin\I$;
\item for any subsets $J\subset I\subset X$ the inclusion $I\in\I$ implies $J\in\I$;
\item for any sets $A,B\in\I$ we have $A\cup B\in\I$.
\end{itemize}
An ideal $\I$ on $X$ is called a {\em $\sigma$-ideal} if for any countable subfamily $\mathcal C\subset\I$ the union $\bigcup\mathcal C$ belongs to $X$.

An ideal $\I$ on a group $X$ is called {\em invariant} if for any $I\in\I$ and $x\in X$ the shift $x+I$ of $I$ belongs to the ideal $\I$.

For example, the family $[X]^{<\w}$ of finite subsets of an infinite group $X$ is an invariant ideal on $X$. 

A topological space is {\em Polish} if it is homeomorphic to a separable complete metric space. A topological space is {\em analytic} if it is a continuous image of a Polish space.

\section{Introducing null-finite sets in topological groups}\label{s:Haar-finite}

In this section we introduce the principal new notion of this paper.

A sequence $(x_n)_{n\in\w}$ in a topological group $X$ is called a {\em null-sequence} if it converges to the neutral element $\theta$ of $X$.

\begin{definition}
A set $A$ of a topological group $X$  is called {\em null-finite} if there exists a null-sequence $(x_n)_{n\in\w}$ in $X$ such that for every $x\in X$ the set $\{n\in\w:x+x_n\in A\}$ is finite.
\end{definition}

Null-finite sets in metrizable topological groups can be defined as follows.

\begin{proposition}\label{p:Hf} For a non-empty subset $A$ of a metric group $X$ the following conditions are equivalent:
\begin{enumerate}
\item $A$ is null-finite;
\item there exists an infinite compact set $K\subset X$ such that for every $x\in X$ the intersection $K\cap(x+A)$ is finite;
\item there exists a continuous map $f:K\to X$ from an infinite compact space $K$ such that for every $x\in X$ the preimage $f^{-1}(A+x)$ is finite.
\end{enumerate}
\end{proposition}

\begin{proof} The implications $(1)\Rightarrow (2)\Rightarrow (3)$ are obvious, so it is enough to prove $(3)\Rightarrow (1)$. Assume that $f:K\to X$ is a continuous map from an infinite compact space $K$ such that for every $x\in X$ the preimage $f^{-1}(A+x)$ is finite. Fix any point $a\in A$. It follows that for every $x\in X$ the fiber $f^{-1}(a+x)\subset f^{-1}(A+x)$ is finite and hence the image $f(K)$ is infinite. So, we can choose a sequence $(y_n)_{n\in\w}$ of pairwise distinct points in $f(K)$. Because of the compactness and metrizability of $f(K)$, we can additionally assume that the sequence $(y_n)_{n\in\w}$ converges to some point $y_\infty\in f(K)$. Then the sequence $(x_n)_{n\in\w}$ of points $x_n:=y_n-y_\infty$, $n\in\w$, is a null-sequence witnessing that the set $A$ is null-finite.
\end{proof}

Null-finite sets are ``small'' in the following sense.

\begin{proposition}\label{p:Fempty} For any null-finite set $A$ in a metric group $X$ and any finite subset $F\subset X$ the set $F+A$ has empty interior in $X$.
\end{proposition}

\begin{proof} Since $A$ is null-finite, there exists a null-sequence $(x_n)_{n\in\w}$ in $X$ such that
for every $x\in X$ the set $\{n\in\w:x+x_n\in A\}$ is finite.

To derive a contradiction, assume that for some finite set $F$ the set $F+A$ has non-empty interior $U$ in $X$. Then for any point $u\in U$ the set $\{n\in\w:u+x_n\in U\}\subset \{n\in\w:u+x_n\in F+A\}$ is infinite. By the Pigeonhole Principle, for some $y\in F$ the set $\{n\in\w:u+x_n\in y+A\}$ is infinite and so is the set $\{n\in\w:-y+u+x_n\in A\}$. But this contradicts the choice of the sequence $(x_n)_{n\in\w}$.
\end{proof}

Now we present some examples of closed sets which are (or are not) null-finite. We recall that for a group $G$ by $G^*$ we denote the set $G\setminus\{\theta\}$ of non-zero elements of $G$.

\begin{example}\label{ex} Let $(G_n)_{n\in\w}$ be a sequence of finite groups. The set $A:=\prod_{n\in\w}G_n^*$ is null-finite in the compact metrizable topological group $G=\prod_{n\in\w}G_n$ if and only if\/  $\lim_{n\to\infty}|G_n|\ne\infty$.
\end{example}

\begin{proof} First, we assume that $\lim_{n\to\infty}|G_n|=\infty$ and show that the set $A$ is not null-finite in $G$. Given a null-sequence $(x_n)_{n\in\w}$ in the compact topological group $G$, we should find an element $a\in G$ such that the set $\{n\in\w:a+x_n\in A\}$ is infinite.

It will be convenient to think of elements of the group $G$ as functions $x:\w\to \bigoplus_{n\in\w}G_n$ such that $x(k)\in G_k$ for all $k\in\w$.

Taking into account that $\lim_{k\to\infty}|G_k|=\infty$ and $\lim_{n\to\infty}x_n=\theta\in G$, we can inductively construct an increasing number sequence $(n_k)_{k\in\w}$ such that $n_0=0$ and for every $k\in\IN$ and $m\ge n_{k}$ the following two conditions are satisfied:
\begin{itemize}
\item[(1)] $x_m(i)=\theta$ for all $i<n_{k-1}$;
\item[(2)] $|G_m|\ge k+3$.
\end{itemize}

We claim that for every $i\ge n_1$ the set $\{x_{n_k}(i):k\in\w\}$ has cardinality $<|G_i|$. Indeed, given any $i\ge n_1$, we can find a unique $j\ge 1$ such that $n_j\le i<n_{j+1}$ and conclude that $x_{n_k}(i)=\theta$ for all $k\ge j+1$ (by the condition (1)). Then the set $\{x_{n_k}(i):k\in\w\}=\{\theta\}\cup\{x_{n_k}(i):k\le j\}$ has cardinality $\le 1+(j+1)<j+3\le |G_i|$. The last inequality follows from $i\ge n_k$ and condition (2). Therefore, the set $\{x_{n_k}(i):k\in\w\}$ has cardinality $<|G_i|$ and we can choose a point $a_i\in G_i\setminus\{-x_{n_k}(i):k\in\w\}$ and conclude that $a_i+x_{n_k}(i)\in G_i^*$ for all $k\in\w$. Then the element $a=(a_i)_{i\in\w}\in G$ has the required property as the set $\{n\in\w:a+x_n\in A\}\supset\{n_k\}_{k\in\w}$ is infinite.
\smallskip

Next, assuming that $\lim_{k\to\infty}|G_k|\ne\infty$, we shall prove that the set $A$ is Haar-finite in $G$. Since $\lim_{k\to\infty}|G_k|\ne\infty$, for some $l\in\w$ the set $\Lambda=\{k\in\w:|G_k|=l\}$ is infinite.

Since $\lim_{n\to\infty}(1-\frac1l)^n=0$, there exists $n\in\IN$ such that $(1-\frac1l)^n<\frac1l$ and hence $l(l-1)^n<l^n$. This implies that for every $k\in\Lambda$ the set $G_k{\cdot}(G_k^*)^n:=\{(x+a_1,\cdots x+a_n):x\in G_k,\;a_1,\dots,a_n\in G_k^*\}$ has cardinality $|G_k{\cdot}(G_k^*)^n|\le l(l-1)^n<l^n=|G_k^n|$. Consequently, the (compact) set $G{\cdot}A^n=\{(x+a_1,\dots,x+a_n):x\in G,\;a_1,\dots,a_n\in A\}$ is nowhere dense in the compact topological group $G^n$. Now Theorem~\ref{t:MK} ensures that $A$ is null-finite in $G$.
\end{proof}

\begin{theorem}\label{t:MK} A non-empty subset $A$ of a  Polish group $X$ is null-finite in $X$ if for some $n\in\IN$ the set $X{\cdot} A^n:=\{(x+a_1,\dots,x+a_n):x\in X,\;a_1,\dots,a_n\in A\}$ is meager in $X^n$.
\end{theorem}

\begin{proof} Let $\K(X)$ be the space of non-empty compact subsets of $X$, endowed with the Vietoris topology. For a compact subset $K\subset X$ let $K^n_*=\{(x_1,\dots,x_n)\in K^n:|\{x_1,\dots,x_n\}|=n\}$ the set of $n$-tuples consisting of pairwise distinct points of $K$.

By Mycielski-Kuratowski Theorem \cite[19.1]{Ke}, the set $W_*(A):=\{K\in\K(X):K^n_*\cap (X{\cdot}A^n)=\emptyset\}$ is comeager in $\K(X)$ and hence contains some infinite compact set $K$. We claim that for every $x\in X$ the set $K\cap (x+A)$ has cardinality $<n$. Assuming the opposite, we could find $n$ pairwise distinct points $a_1,\dots,a_n\in A$ such that $x+a_i\in K$ for all $i\le n$. Then $(x+a_1,\dots,x+a_n)\in K^n_*\cap X\cdot A^n$, which contradicts the inclusion $K\in W_*(A)$. By Proposition~\ref{p:Hf}, the set $A$ is null-finite.
\end{proof}

\section{A Steinhaus-like properties of sets which are not null-finite}\label{s:Steinhaus}

In \cite{Steinhaus} Steinhaus proved that for a subset $A$ of positive Lebesgue measure in the real line, the set $A-A$ is a neighborhood of zero.
In this section we establish three Steinhaus-like properties of sets which are not null-finite.

\begin{theorem}\label{t:Steinhaus} If a subset $A$ of a metric group $X$ is not null-finite, then
\begin{enumerate}
\item the set $A-\bar A$ is a neighborhood of $\theta$ in $X$;
\item each neighborhood $U\subset X$ of $\theta$ contains a finite subset $F\subset U$ such that $F+(A-A)$ is a neighborhood of $\theta$.
\item If $X$ is Polish (and $A$ is analytic), then $A-A$ is not meager in $X$ (and $(A-A)-(A-A)$ is a neighborhood of $\theta$ in $X$).
\end{enumerate}
\end{theorem}

\begin{proof} 1. Assuming that $A-\bar A$ is not a neighborhood of $\theta$, we could find a null-sequence $(x_n)_{n\in\w}$ contained in $X\setminus(A-\bar A)$. Since $A$ is not null-finite, there exists $a\in X$ such that the set $\Omega=\{n\in\w:a+x_n\in A\}$ is infinite. Then $a\in\overline{\{a+x_n\}}_{n\in\Omega}\subset\bar A$ and hence $x_n=(a+x_n)-a\in A-\bar A$ for all $n\in\Omega$, which contradicts the choice of the sequence $(x_n)_{n\in\w}$.
\smallskip

2. Fix a decreasing neighborhood base $(U_n)_{n\in\w}$ at zero in $X$ such that $U_0\subset U$. For the proof by contradiction,  suppose that for any finite set $F\subset U$ the set $F+(A-A)$ is not a neighborhood of zero. Then  we can inductively construct a null-sequence $(x_n)_{n\in\w}$ such that $x_n\in U_n\setminus \bigcup_{0\le i<n}(x_i+A-A)$ for all $n\in\w$. Observe that for each $z\in X$ the set $\{n\in\w:z+x_n\in A\}$ contains at most one point. Indeed, in the opposite case we could find two numbers $k<n$ with $z+x_k,z+x_n\in A$ and conclude that $z\in -x_k+A$ and hence $x_n\in -z+A\subset x_k-A+A$, which contradicts the choice of the number $x_n$. The sequence $(x_n)_{n\in\w}$ witnesses that the set $A$ is null-finite in $X$.
\smallskip

3. Finally assume that the group $X$ is Polish. To derive a contradiction, assume that $A-A$ is meager. Since the map $\delta:X\times X\to X$, $\delta:(x,y)\mapsto x-y$, is open, the preimage $\delta^{-1}(A-A)$ is meager in $X\times X$. Observe that $\delta^{-1}(A-A)=X{\cdot}A^2$ where $X{\cdot}A^2=\{(x+a,x+b):x\in X,\;a,b\in A\}$. By Theorem~\ref{t:MK}, the set $A$ is null-finite, which is a desired contradiction. This contradiction shows that the set $A-A$ is not meager in $X$. If the space $A$ is analytic, then so is the subset $A-A$ of $X$. Being analytic, the set $A-A$ has the Baire Property in $X$; see \cite[29.14]{Ke}. So, we can apply the Pettis-Piccard Theorem \cite[9.9]{Ke} and conclude that $(A-A)+(A-A)$ is a neighborhood of $\theta$ in $X$.
\end{proof}

The last statement of Theorem~\ref{t:Steinhaus} does not hold without any assumptions on $A$ and $X$.

\begin{example} {\rm Let $f:C_2^\w\to C_2$ be a discontinuous homomorphism on the compact Polish group $X=C_2^\w$ (thought as the vector space $C_2^\w$ over the two-element field $\IZ/2\IZ$). Then the subgroup $A=f^{-1}(\theta)$ is not null-finite in $X$ (by Proposition~\ref{p:Fempty}), but $A-A=A=(A-A)-(A-A)$ is not a neighborhood of $\theta$ in $X$.}
\end{example}

\section{A combinatorial characterization of null-finite sets in compact metric groups}

In compact metric groups null-finite sets admit a purely combinatorial description.

\begin{proposition}\label{p:char-comp} A non-empty subset $A$ of a compact metric group $X$ is null-finite if and only if there exists an infinite set $I\subset X$ such that for any infinite subset $J\subset I$ the intersection $\bigcap_{x\in J}(A-x)$ is empty.
\end{proposition}

\begin{proof} To prove the ``only if'' part, assume that $A\subset X$ is null-finite. So, there exists a null-sequence $(x_n)_{n\in\w}$ such that for every $x\in X$ the set $\{n\in\w:x+x_n\in A\}$ is finite. It follows that for every $x\in X$ the set $\{n\in\w:x_n=x\}$ is finite and hence the set $I:=\{x_n\}_{n\in\w}$ is infinite. We claim that this set has the required property. Indeed, assuming that for some infinite subset $J\subset I$ the intersection $\bigcap_{x\in J}(A-x)$ contains some point $a\in X$, we conclude that the set $\{n\in\w:a+x_n\in A\}$ contains the set $\{n\in\w:x_n\in J\}$ and hence is infinite, which contradicts the choice of the sequence $(x_n)_{n\in\w}$.
\smallskip

To prove the ``if'' part, assume that there exists an infinite set $I\subset X$ such that for every infinite set $J\subset I$ the intersection $\bigcap_{x\in J}(A-x)$ is empty. By the compactness of the metric group $X$, some sequence $(x_n)_{n\in\w}$ of pairwise distinct points of the infinite set $I$ converges to some point $x_\infty\in X$. Then the null-sequence $(z_n)_{n\in\w}$ consisting of the points $z_n=x_n-x_\infty$, $n\in\w$, witnesses that $A$ is null-finite. Assuming the opposite, we would find a point $a\in X$ such that the set $\{n\in\w:a+z_n\in A\}$ is infinite. Then the set $J:=\{x\in I:a-x_\infty+x\in A\}\subset \{x_n:n\in\w,\; a-x_\infty+x_n\in A\}$ is infinite, too, and the intersection $\bigcap_{x\in J}(A-x)$ contains the point $a-x_\infty$ and hence is not empty, which contradicts the choice of the set $I$.
\end{proof}

Following Lutsenko and Protasov \cite{LP} we define a subset $A$ of an infinite group $X$ to be {\em sparse} if for any infinite set $I\subset X$ there exists a finite set $F\subset I$ such that $\bigcap_{x\in F}(x+A)$ is empty. By \cite[Lemma~1.2]{LP} the family of sparse subsets of a group is an invariant ideal on $X$ (in contrast to the family of null-finite sets).

\begin{proposition} Each sparse subset $A$ of a non-discrete metric group $X$ is null-finite.
\end{proposition}

\begin{proof} Being non-discrete, the metric group $X$ contains a null-sequence $(x_n)_{n\in\w}$ consisting of pairwise distinct points.
Assuming that the sparse set $A$ is not null-finite, we can find a point $a\in X$ such that the set $\Omega:=\{n\in\w:a+x_n\in A\}$ is infinite.
Since $A$ is sparse, for the infinite set $I:=\{-x_n:n\in\Omega\}$ there exists a finite set $F\subset I$ such that the intersection $\bigcap_{x\in F}(x+A)\supset\bigcap_{m\in \Omega}(-x_n+A)\ni a$ is empty, which is not possible as this intersection contains the point $a$.
\end{proof}

\section{Null-finite Borel sets are Haar-meager}\label{s:Haar-meager}

In this section we prove that each null-finite set with the universal Baire property in a complete metric group is Haar-meager.

A subset $A$ of a topological group $X$ is defined to have the {\em universal Baire Property} (briefly, $A$ is a {\em uBP-set}) if for any function $f:K\to X$ from a compact metrizable space $K$ the preimage $f^{-1}(A)$ has the Baire Property in $K$, which means that for some open set $U\subset K$ the symmetric difference $U\triangle f^{-1}(A)$ is meager in $K$. It is well-known that each Borel subset of a topological group has the universal Baire Property.

A uBP-set $A$ of a topological group $X$ is called {\em Haar-meager} if there exists a continuous map $f:K\to X$ from a compact metrizable space $K$ such that $f^{-1}(x+A)$ is meager in $K$ for all $x\in X$. By \cite{Darji}, for a complete metric group $X$ the family $\HM_X$ of subsets of Haar-meager uBP-sets in $X$ is an invariant $\sigma$-ideal on $X$. For more information on Haar-meager sets, see \cite{Darji}, \cite{DVVR}, \cite{DV}, \cite{Jab}.

\begin{theorem}\label{t:Hm} Each null-finite uBP-set in a compete metric group is Haar-meager.
\end{theorem}

\begin{proof} To derive a contradiction, suppose that a null-finite uBP-set $A$ in a complete metric group $(X,\|\cdot\|)$ is not Haar-meager. Since $A$ is null-finite, there exists a~null-sequence $(a_n)_{n\in\w}$ such that for every $x\in X$ the set $\{n\in\w:x+a_n\in A\}$ is finite.

Replacing $(a_n)_{n\in\w}$ by a suitable subsequence, we can assume that $\|a_n\|<\frac{1}{2^n}$ for all $n\in\w$. For every $n\in\w$ consider the compact set $K_n:=\{\theta\}\cup\{a_m\}_{m\ge n}\subset X$. The metric restriction on the sequence $(a_n)_{n\in\w}$ implies that the function $$\Sigma:\prod_{n\in\w}K_n\to X,\;\;\Sigma:(x_n)_{n\in\w}\mapsto\sum_{n\in\w}x_n,$$ is well-defined and continuous (the proof of this fact can be found in \cite{Jab}; see the proof of Theorem 2).

Since the set $A$ is not Haar-meager and has uBP, there exists a point $z\in X$ such that $B:=\Sigma^{-1}(z+A)$ is not meager and has the Baire Property in the compact metrizable space $K:=\prod_{n\in\w}K_n$. Consequently, there exits a non-empty open set $U\subset K$ such that $U\cap B$ is comeager in $U$. Replacing $U$ by a smaller subset, we can assume that $U$ is of basic form $\{b\}\times\prod_{m\ge j}K_m$ for some $j\in\w$ and some element $b\in\prod_{m<j}K_m$. It follows that for every $x\in K_j\setminus\{\theta\}$ the set $\{b\}\times\{x\}\times\prod_{m>j}K_m$ is closed and open in $K$ and hence the set $C_x:= \{y\in\prod_{m>j}K_m:\{b\}\times\{x\}\times\{y\}\subset U\cap B\}$ is comeager in $\prod_{m>j}K_m$. Then the intersection $\bigcap_{x\in K_n\setminus\{\theta\}}C_x$ is comeager and hence contains some point $c$. For this point we get the inclusion $\{b\}\times(K_j\setminus\{\theta\})\times\{c\}\subset B$. Let $v\in K$ be a unique point such that $\{v\}=\{b\}\times\{\theta\}\times \{c\}$. For any $m\ge j$ the point $a_m$ belongs to $K_j$ and the inclusion $\{b\}\times\{a_m\}\times\{c\}\subset B$ implies that $\Sigma(v)+a_m\in z+A$. Consequently, the element $u:=-z+\Sigma(v)$ has the property $u+a_m\in A$ for all $m\ge j$, which implies that the set $\{n\in\w:u+a_n\in A\}$ is infinite.
But this contradicts the choice of the sequence $(a_n)_{n\in\w}$.
\end{proof}

\section{Null-finite Borel sets are Haar-null}\label{s:Haar-null}

In this section we prove that each universally measurable null-finite set in a complete metric group is Haar-null.

A subset $A$ of a topological group $X$ is defined to be {\em universally measurable} if $A$ is measurable with respect to any $\sigma$-additive Radon Borel probability measure on $X$. It is clear that each Radon Borel subset of a topological group is universally measurable. We recall that a measure $\mu$ on a topological space $X$ is {\em Radon} if for any $\e>0$ there exists a compact subset $K\subset X$ such that $\mu(X\setminus K)<\e$.

A universally measurable subset $A$ of a topological group $X$ is called {\em Haar-null} if there exists a $\sigma$-additive Radon Borel probability measure $\mu$ on $X$ such that $\mu(x+A)=0$ for all $x\in X$.
By \cite{Ch}, for any complete metric group $X$ the family $\HN_X$ of subsets of universally measurable Haar-null subsets of $X$ is an invariant $\sigma$-ideal in $X$.
For more information on Haar-null sets, see \cite{Ch}, \cite{HSY}, \cite{HSY1} and also e.g. \cite{EVid}, \cite{FS}, \cite{MZ}, \cite{Solecki}.

\begin{theorem}\label{t:Hn} Each universally measurable null-finite set $A$ in a complete metric group $X$ is Haar-null.
\end{theorem}

\begin{proof} To derive a contradiction, suppose that the null-finite set $A$ is not Haar-null in $X$. Since $A$ is null-finite, there exists a null-sequence $(a_n)_{n\in\w}$ such that for each $x\in X$ the set $\{n\in \w:x+a_n\in A\}$ is finite. Replacing $(a_n)_{n\in\w}$ by a suitable subsequence, we can assume that $\|a_n\|\le\frac1{4^n}$ for all $n\in\w$.

Consider the compact space $\Pi:=\prod_{n\in\w}\{0,1,\dots,2^n\}$ endowed with the product measure $\lambda$ of uniformly distributed measures on the finite discrete spaces $\{0,1,2,3,\dots,2^n\}$, $n\in\w$. Consider the map
$$\Sigma:\Pi\to X,\;\;\Sigma:(p_i)_{i\in\w}\mapsto\sum_{i=0}^\infty p_ia_i.$$ Since $\|p_ia_i\|\le 2^i\|a_i\|\le2^i\frac1{4^i}=\frac1{2^i}$, the series $\sum_{i=0}^\infty p_ia_i$ is convergent and the function $\Sigma$ is well-defined and continuous.

Since the set $A$ is not Haar-null, there exists an element $z\in X$ such that the preimage $\Sigma^{-1}(z+A)$ has positive $\lambda$-measure and hence contains a compact subset $K$ of positive measure. For every $n\in\w$ consider the subcube $\Pi_n:=\prod_{i=0}^{n-1}\{0,1,2,\dots,2^i\}\times \prod_{i=n}^\infty\{1,2,\dots,2^i\}$ of $\Pi$ and observe that $\lambda(\Pi_n)\to1$ as $n\to \infty$. Replacing $K$ by $K\cap \Pi_l$ for a sufficiently large $l$, we can assume that $K\subset\Pi_l$.
For every $m\ge l$ let $s_m:\Pi_l\to\Pi$ be the ``back-shift" defined by the formula $s_m((x_i)_{i\in\w}):=(y_i)_{i\in\w}$ where $y_i=x_i$ for $i\ne m$ and $y_i=x_i-1$ for $i=m$.

\begin{claim}\label{claim} For any compact set $C\subset\Pi_l$ of positive measure
 $\lambda(C)$ and any $\e>0$ there exists $k\ge l$ such that for any $m\ge k$ the intersection $C\cap s_m(C)$ has measure $\lambda(C\cap s_m(C))>(1-\e)\lambda(C)$.
\end{claim}

\begin{proof} By the regularity of the measure $\lambda$, the set $C$ has a neighborhood $O(C)\subset\Pi$ such that $\lambda(O(C)\setminus C)<\e\lambda(K)$. By the compactness of $C$, there exists $k\ge l$ such that for any $m\ge k$ the shift $s_m(C)$ is contained in $O(C)$. Hence $\lambda(s_m(C)\setminus C)\le \lambda(O(C)\setminus C)<\e\lambda(C)$
and thus $\lambda(s_m(C)\cap C)=\lambda(s_m(C))-\lambda(s_m(C)\setminus C)>\lambda(C)-\e\lambda(C)=(1-\e)\lambda(C)$.
\end{proof}

Using Claim~\ref{claim} we can choose an increasing number sequence $(m_k)_{k\in\w}$ such that $m_0>l$ and the set $K_\infty:=\bigcap_{k\in\w}s_{m_k}(K)$ has positive measure and hence contains a point $\vec b:=(b_i)_{i\in\w}$. It follows that for every $k\in\w$ the point $\vec b_k:=s_{m_k}^{-1}(\vec b)$ belongs to
$K\subset \Sigma^{-1}(z+A)$.

Observe that $\Sigma(\vec b_k)=\Sigma(\vec b)+a_{m_k}\in z+A$ and hence $-z+\Sigma(\vec b)+a_{m_k}\in A$ for all $k\in\w$, which contradicts the choice of the sequence $(a_n)_{n\in\w}$.
\end{proof}

\begin{remark} After writing the initial version of this paper we discovered that a result similar to Theorem~\ref{t:Hn} was independently found by Bingham and Ostaszewski \cite[Theorem 3]{BO1}.
\end{remark}

\section{The $\sigma$-ideal generated by (closed) Borel null-finite sets}

In this section we introduce two new invariant $\sigma$-ideals generated by (closed) Borel null-finite sets and study the relation of these new ideals to the $\sigma$-ideals of Haar-null and Haar-meager sets.

Namely, for a complete metric group $X$
let
\begin{itemize}
\item $\sigma\NF_X$ be the smallest $\sigma$-ideal containing all Borel null-finite sets in $X$;
\item $\sigma\overline{\NF}_X$ be the smallest $\sigma$-ideal containing all closed null-finite sets in $X$;
\item $\sigma\overline{\HN}_X$ be the smallest $\sigma$-ideal containing all closed Haar-null sets in $X$.
\end{itemize}
Theorems~\ref{t:Hm} and \ref{t:Hn} imply that $\sigma\overline{\NF}_X\subset\sigma\overline{\HN}_X$ and $\sigma\NF_X\subset\HN_X\cap\HM_X$.
So, we obtain the following diagram in which an arrow $\mathcal A\to\mathcal B$ indicates that $\mathcal A\subset\mathcal B$.
$$
\xymatrix{
{\sigma\overline{\NF}_X}\ar[r]\ar[d]&{\sigma\overline{\HN}_X}\ar[d]\\
{\sigma\NF_X}\ar[r]&{\HN_X\cap\HM_X}
}
$$ In Examples~\ref{ex1} and \ref{ex2} we show that the $\sigma$-ideal $\sigma\overline{\NF}_X$ is strictly smaller than $\sigma\overline{\HN}_X$ and the ideal $\sigma\NF_X$ is not contained in $\sigma\overline{\HN}_X$.

\begin{example}\label{ex1} The closed set $A=\prod_{n\ge 2}C_n^*$ in the product $X=\prod_{n\ge 2}C_n$ of  cyclic groups is Haar-null but cannot be covered by countably many closed null-finite sets. Consequently, $A\in\sigma\overline{\HN}_X\setminus\sigma\overline{\NF}_X$.
\end{example}

\begin{proof} The set $A=\prod_{n=2}^\infty C_n^*$ has Haar measure $$\prod_{n=2}^\infty\frac{|C_n^*|}{|C_n|}=\prod_{n=2}^\infty\frac{n-1}n=0$$
and hence is Haar-null in the compact Polish group $X$.

Next, we show that $A$ cannot be covered by countably many closed null-finite sets. To derive a contradiction, assume that $A=\bigcup_{n\in\w}A_n$ where each $A_n$ is closed and null-finite in $X$. By the Baire Theorem, for some $n\in\w$ the set $A_n$ has non-empty interior in $A$. Consequently, we can find $m>2$ and $a\in\prod_{n=2}^{m-1}C_n^*$ such that $\{a\}\times\prod_{n=m}^\infty C_n^*\subset A_n$. By Example~\ref{ex}, the set $\prod_{n=m}^\infty C_n^*$ is not null-finite in $\prod_{n=m}^\infty C_n$, which implies that the set $A_n\supset \{a\}\times\prod_{n=m}^\infty C_n^*$ is not null-finite in the group $X$. But this contradicts the choice of $A_n$.
\end{proof}

Our next example shows that $\sigma\NF_X\not\subset \sigma\overline{\HN}_X$ for some compact Polish group $X$.

\begin{example}\label{ex2} For any function $f:\w\to[2,\infty)$ with $\prod_{n\in\w}\frac{f(n)-1}{f(n)}>0$, the compact metrizable group $X=\prod_{n\in\IN}C_{f(n)}$ contains a null-finite $G_\delta$-set $A\subset X$ which cannot be covered by countably many closed Haar-null sets in $X$. Consequently, $A\in\NF_X\setminus\sigma\overline{\HN}_X$.
\end{example}

\begin{proof} For every $n\in\w$ let $g_n$ be a generator of the cyclic group $C_{f(n)}$. In the compact metrizable group $X=\prod_{n\in\w}C_{f(n)}$ consider the null-sequence $(x_n)_{n\in\w}$ defined by the formula $$x_n(i)=\begin{cases}\theta&\mbox{if $i\le n$}\\
g_i&\mbox{if $i>n$}.
\end{cases}
$$
Consider the closed subset $B=\prod_{n\in\w}C_{f(n)}^*$ in $X$ and observe that it has positive Haar measure, equal to the infinite product $$\prod_{n\in\w}\frac{f(n)-1}{f(n)}>0.$$
It is easy to see that for every $n\in\w$ the set $C_{f(n)}^*=C_{f(n)}\setminus\{\theta\}$ is not equal to $C_{f(n)}^*\cap (g_n+C_{f(n)}^*)$, which implies that the intersection $B\cap (x_n+B)$ is nowhere dense in $B$. Consequently, the set  $A=B\setminus\bigcup_{n\in\w}(x_n+B)$ is a dense $G_\delta$-set in $B$.

We claim that the set $A$ is null-finite. Given any $a\in X$ we should prove that the set $\{n\in\w:a+x_n\in A\}$ is finite. If $a\notin B$, then the open set $X\setminus B$ is a neighborhood of $a$ in $X$. Since the sequence $(a+x_n)_{n\in\w}$ converges to $a\in X\setminus B$, the set $\{n\in\w:a+x_n\in B\}\supset\{n\in\w:a+x_n\in A\}$ is finite.

If $a\in B$, then $\{n\in\w:a+x_n\in A\}\subset \{n\in\w:(B+x_n)\cap A\ne\emptyset\}=\emptyset$ by the definition of the set $A$.

Next, we prove that the $G_\delta$-set $A$ cannot be covered by countably many closed Haar-null sets. To derive a contradiction, assume that $A\subset \bigcup_{n\in\w}F_n$ for some closed Haar-null sets $F_n\subset X$. Since the space $A$ is Polish, we can apply the Baire Theorem and find $n\in\w$ such that the set $A\cap F_n$ has non-empty interior in $A$ and hence its closure $\overline{A\cap F_n}$ has non-empty interior in $B=\prod_{n\in\w}C_{f(n)}^*$. It is easy to see that each non-empty open subset of $B$ has positive Haar measure in $X$. Consequently, the set $\overline{A\cap F_n}$ has positive Haar measure, which is not possible as this set is contained in the Haar-null set $F_n$.
\end{proof}

\begin{remark} Answering a question posed in a preceding version of this paper, Adam Kwela \cite{Kwela} constructed two compact null-finite subsets  $A,B$ on the real line, whose union $A\cup B$ is not null-finite. This means that the family of subsets of (closed) Borel null-finite subsets on the real line is not an ideal, and the ideal $\sigma\overline{\NF}_\IR$ contains compact subsets of the real line, which fail to be null-finite.
\end{remark}

\section{Decomposing non-discrete metric groups into unions of two null-finite sets}\label{s:union}

By Theorem~\ref{t:Hn} and the countable additivity of the family of Borel Haar-null sets \cite{Ch}, the countable union of Borel null-finite sets in a complete metric group $X$ is Haar-null in $X$ and hence is not equal to $X$. So, a complete metric group cannot be covered by countably many Borel null-finite sets. This result dramatically fails for non-Borel null-finite sets.

\begin{theorem}\label{t:union} Each non-discrete metric group $X$ can be written as the union $X=A\cup B$ of two null-finite subsets $A,B$ of $X$.
\end{theorem}

\begin{proof}
Being non-discrete, the metric group $X$ contains a non-trivial null-sequence, which generates a non-discrete countable subgroup $Z$ in $X$.

Let $Z=\{z_n\}_{n\in\w}$ be an enumeration of the countable infinite group $Z$ such that $z_0=\theta$ and $z_n\ne z_m$ for any distinct numbers $n,m\in\w$. By induction we can construct sequences $(u_n)_{n\in\w}$ and $(v_n)_{n\in\w}$ in $Z$ such that $u_0=v_0=\theta$ and for every $n\in\IN$ the following two conditions are satisfied:
\begin{enumerate}
\item $\|u_n\|\le\frac1n$ and $u_n\notin\{-z_i+z_j+v_k:i,j\le n,\;k<n\}\cup\{-z_i+z_j+u_k:i,j\le n,\;k<n\}$;
\item $\|v_n\|\le\frac1n$ and $v_n\notin\{-z_i+z_j+u_k:i,j\le n,\;k\le n\}\cup\{-z_i+z_j+v_k:i,j\le n,\;k<n\}$.
\end{enumerate}
At the $n$-th step of the inductive construction, the choice of the
points $u_n,v_n$ is always possible as the ball $\{z\in Z:\|z\|\le\frac1n\}$ is infinite and $u_n,v_n$ should avoid finite sets.

After completing the inductive construction, we obtain the null-sequences $(u_n)_{n\in\w}$ and $(v_n)_{n\in\w}$ of pairwise distinct points of $Z$ such that for any points $x,y\in Z$ the intersection $\{x+u_n\}_{n\in\w}\cap \{y+u_m\}_{m\in\w}$ is finite and for every distinct points $x,y\in Z$ the intersections $\{x+u_n\}_{n\in\w}\cap\{y+u_m\}_{m\in\w}$ and $\{x+v_n\}_{n\in \w}\cap\{y+v_m\}_{m\in\w}$ are finite.

Using these facts, for every $n\in\w$ we can choose a number $i_n\in\w$ such that the set $\{z_n+v_k:k\ge i_n\}$ is disjoint with the set $\{z_i+v_m:i<n,\;m\in\w\}\cup\{z_i+u_m:i\le n,\;m\in\w\}$ and the set $\{z_n+u_k:k\ge i_n\}$ is disjoint with the set $\{z_i+u_m:i<n,\;m\in\w\}\cup\{z_i+v_m:i\le n,\;m\in\w\}$.

We claim that the set $A:=\bigcup_{n\in\w}\{z_n+u_m:m>i_n\}$ is null-finite in $Z$. This will follow as soon as we verify the condition:
\begin{enumerate}
\item[(3)] for any $z\in Z$ the set $\{n\in\w:z+v_n\in A\}$ is finite.
\end{enumerate}
Find $j\in\w$ such that $z=z_j$ and observe that for every $n\ge j$ the choice of the number $i_n$ guarantees that $\{z_n+u_m\}_{m> i_n}\cap\{z_j+v_n\}_{n\in\w}=\emptyset$, so $\{n\in\w:z+v_n\in A\}$ is contained in the finite set $\bigcup_{n<j}\{z_n+v_m\}_{m\in\w}\cap\{z_j+u_m\}_{m\in\w}$ and hence is finite.

Next, we show that the set $B:=Z\setminus A$ is null-finite in $Z$. This will follow as soon as we verify the condition
\begin{enumerate}
\item[(4)] for any $z\in Z$ the set $\{m\in\w:z+u_m\in B\}$ is finite.
\end{enumerate}
Find a number $n\in\w$ such that $z=z_n$ and observe that $\{m\in\w:z+u_m\in B\}=\{m\in\w:z_n+u_m\notin A\}\subset\{m\in\w:m\le i_n\}$ is finite.
Therefore, the sets $A$ and $B=Z\setminus A$ are null-finite in $Z$.
\smallskip

Using Axiom of Choice, choose a subset $S\subset X$ such that for every $x\in X$ the intersection $S\cap(x+Z)$ is a singleton. It follows that $X=S+Z=(S+A)\cup(S+B)$. We claim that the sets $S+A$ and $S+B$ are null-finite in $X$. This will follow as soon as for any $x\in X$ we check that the sets $\{n\in\w:x+v_n\in S+A\}$ and $\{n\in\w:x+u_n\in S+B\}$ are finite. Since $X=S+Z$, there exist elements $s\in S$ and $z\in Z$ such that $x=s+z$. Observe that if for $n\in\w$ we get $x+v_n\in S+A$, then $s+z+v_n=t+a$ for some $t\in S$ and $a\in A\subset Z$.
Subtracting $a$ from both sides, we get $t=s+z+v_n-a\in s+Z$ and hence $t=s$ as $S\cap(s+Z)$ is a singleton containing both points $t$ and $s$.
So, the inclusion $x+v_n=s+z+v_n\in S+A$ is equivalent to $z+v_n\in A$. By analogy we can prove that $x+u_n\in S+B$ is equivalent to $z+u_n\in B$.
Then the sets
$$
\{n\in\w:x+v_n\in S+A\}=\{n\in\w:z+v_n\in A\}\mbox{ \  and \ }
\{n\in\w:x+u_n\in S+B\}=\{n\in\w:z+u_n\in B\}
$$ are finite by the properties (3) and (4) of the null-finite sets $A,B$.
\end{proof}

Proposition~\ref{p:Fempty} and Theorem~\ref{t:union} imply the following corollary.

\begin{corollary} Let $X$ be a non-discrete metric group.
\begin{enumerate}
\item Each null-finite subset  of $X$ is contained in some invariant ideal on $X$;
\item For any ideal $\I$ on $X$ there exists a null-finite set $A\subset X$ such that $A\notin \I$.
\end{enumerate}
\end{corollary}

\begin{proof}
1. For any null-finite set $A$ in $X$ the family $$\I_A=\{I\subset X:\exists F\in[X]^{<\w}\;I\subset FA\}$$ is an invariant ideal on $X$ whose elements have empty interiors in $X$ by Proposition~\ref{p:Fempty}. We recall that by $[X]^{<\w}$ we denote the ideal of finite subsets of $X$.
\smallskip

2. By Theorem~\ref{t:union}, the group $X$ can be written as the union $X=A\cup B$ of two null-finite sets. Then for any ideal $\I$ on $X$ one of the sets $A$ or $B$ does not belong to $\I$.
\end{proof}

\section{Applying null-finite sets to additive functionals}\label{s:functional}

In this section we apply null-finite sets to prove a criterion of continuity of additive functionals on metric groups.

A function $f:X\to Y$ between groups is called {\em additive} if $f(x+y)=f(x)+f(y)$ for every $x,y\in X$. An additive function into the real line is called an {\em additive functional}.

\begin{theorem}\label{t1} An additive functional $f:X\to\IR$ on a metric group $X$ is continuous if it is upper bounded on a set $B\subset X$ which is not null-finite.
\end{theorem}

\begin{proof} Suppose that the functional  $f$ is not
continuous. Then there exists $\e>0$ such that $f(U)\not\subset (-\e,\e)$ for each neighborhood $U\subset X$ of zero. It follows that for every $n\in\w$ there is a point $x_n\in X$ such that $\|x_n\|\le \frac{1}{4^n}$ and $|f(x_n)|>\e$. Observe that, $\|2^nx_n\|\le 2^n\cdot\|x_n\|\le\frac1{2^n}$ and $|f(2^nx_n)|=2^n|f(x_n)|>2^n\e$. Choose $\e_n\in\{1,-1\}$ such that $\e_nf(2^nx_n)$ is positive and put $z_n:=\e_n2^nx_n$. Then $f(z_n)>2^n\e$ and $\|z_n\|\le\frac1{2^n}$.

We claim that the null-sequence $(z_n)_{n\in\w}$ witnesses that the set $B$ is null-finite. Given any point $x\in X$ we need to check that the set $\Omega:=\{n\in\w:x+z_n\in B\}$ is finite.

Let $M:=\sup f(B)$ and observe that for every $n\in \Omega$ we have $$M\ge f(x+z_n)=f(x)+f(z_n)\ge f(x)+2^n\e,$$ which implies $2^n\le \frac1{\e}(M-f(x))$. Hence the set $\Omega$ is finite. Therefore, the null-sequence $(z_n)_{n\in\w}$ witnesses that the set $B$ is null-finite, which contradicts our assumption.
\end{proof}

Theorems~\ref{t1} and \ref{t:Hm} imply

\begin{corollary} An additive functional $f:X\to\IR$ on a complete metric group is continuous if and only if it is upper bounded on some uBP set $B\subset X$ which is not Haar-meager.
\end{corollary}

An analogous result for Haar-null sets follows from Theorems~\ref{t1} and \ref{t:Hn}.

\begin{corollary}\label{c2} An additive functional $f:X\to\IR$ on a complete metric group is continuous if and only if it is upper bounded on some universally measurable set $B\subset X$ which is not Haar-null.
\end{corollary}

\section{Applying null-finite sets to additive functions}\label{s:Banach}

In this section we prove some continuity criteria for additive functions with values in Banach spaces or locally convex spaces.

\begin{corollary}\label{c:Banach} An additive function $f:X\to Y$ from a complete metric group $X$ to a  Banach space $Y$ is continuous if for any linear continuous functional $y^*:Y\to\IR$ the function $y^*\circ f:X\to\IR$ is upper bounded on some set $B\subset X$ which is not null-finite in $X$.
\end{corollary}

\begin{proof} Assuming that $f$ is discontinuous, we can find $\e>0$ such that for each neighborhood $U\subset X$ of $\theta$ we get $f(U)\not\subset B_\e:=\{y\in Y:\|y\|<\e\}$. Then for every $n\in\w$ we can find a point $x_n\in X$ such that $\|x_n\|\le\frac1{4^n}$ and $\|f(x_n)\|\ge \e$. Since $\|2^nx_n\|\le 2^n\|x_n\|\le \frac1{2^n}$, the sequence $(2^nx_n)_{n\in\w}$ is a null-sequence in $X$. On the other hand, $\|f(2^nx_n)\|=2^n\|f(x_n)\|\ge 2^n\e$ for all $n\in\w$, which implies that the set $\{f(2^nx_n)\}_{n\in\w}$ is unbounded in the Banach space $Y$. By the Banach-Steinhaus Uniform Boundedness Principle \cite[3.15]{FA}, there exists a linear continuous functional $y^*:Y\to\IR$ such that the set $\{y^*\circ f(2^nx_n):n\in\w\}$ is unbounded in $\IR$.
By our assumption, the additive functional $y^*\circ f:X\to\IR$ is upper bounded on some set $B\subset X$ which is not null-finite in $X$. By Theorem~\ref{t1}, the additive functional $y^*\circ f$ is continuous and thus the sequence $\big(y^*\circ f(2^nx_n)\big)_{n\in\w}$ converges to zero and hence cannot be unbounded in $\IR$. This is a desired contradiction, completing the proof.
\end{proof}

Corollary~\ref{c:Banach} admits a self-generalization.

\begin{theorem}\label{t:lc} An additive function $f:X\to Y$ from a complete metric group $X$ to a locally convex space $Y$ is continuous if for any linear continuous functional $y^*:Y\to\IR$ the function $y^*\circ f:X\to\IR$ is upper bounded on some set $B\subset X$ which is not null-finite in $X$.
\end{theorem}

\begin{proof} This theorem follows from Corollary~\ref{c:Banach} and the well-known fact \cite[p.54]{Sch} that each locally convex space is topologically isomorphic to a linear subspace of a Tychonoff product of Banach spaces.
\end{proof}

Combining Theorem~\ref{t:lc} with Theorems~\ref{t:Hm} and \ref{t:Hn}, we derive two corollaries.

\begin{corollary} An additive function $f:X\to Y$ from a complete metric group $X$ to a locally convex space $Y$ is continuous if for any linear continuous functional $y^*:Y\to\IR$ the function $y^*\circ f:X\to\IR$ is upper bounded on some uBP-set $B\subset X$ which is not Haar-meager in $X$.
\end{corollary}

\begin{corollary} An additive function $f:X\to Y$ from a complete metric group $X$ to a locally convex space $Y$ is continuous if for any linear continuous functional $y^*:Y\to\IR$ the function $y^*\circ f:X\to\IR$ is upper bounded on some universally measurable set $B\subset X$ which is not Haar-null in $X$.
\end{corollary}

\section{Applying null-finite sets to mid-point convex functions}\label{s:Jensen}

In this section we apply null-finite sets to establish a continuity criterion for mid-point convex functions on linear metric spaces.

A function $f:C\to \IR$ on a convex subset $C$ of a linear space is called {\em mid-point convex} if $$f(\tfrac{x}2+\tfrac{y}2)\le\tfrac12{f(x)}+\tfrac12{f(y)}$$ for any $x,y\in C$. Mid-point convex functions are alternatively called {\em Jensen convex}.

\begin{theorem}\label{t:J} A mid-point convex function $f:G\to\IR$ defined on an open convex subset $G\subset X$ of a metric linear space $X$ is continuous if and only if $f$ is upper bounded on some  set $B\subset G$ which is not null-finite in $X$ and whose closure $\overline{B}$ is contained in $G$.
\end{theorem}

\begin{proof} The ``only if'' part is trivial. To prove the ``if'' part, assume that $f$ is upper bounded on some set $B\subset G$ such that $\overline{B}\subset G$ and $B$ is not null-finite in $X$. We need to check the continuity of $f$ at any point $c\in G$. Shifting the set $G$ and the function $f$ by $-c$, we may assume that $c=\theta$. Also we can replace the function $f$ by $f-f(\theta)$ and assume that $f(\theta)=0$. In this case the mid-point convexity of $f$ implies that $f(\frac1{2^n}x)\le \frac1{2^n}f(x)$ for every $x\in G$ and $n\in\IN$.

To derive a contradiction, suppose that the function  $f$ is not
continuous at $\theta$.  Then there exists $\e>0$ such that $f(U)\not\subset (-\e,\e)$ for any open neighborhood $U\subset G$.
For every $n\in\w$ consider the neighborhood $U_n:=\{x\in G:\|x\|\le \frac1{4^n},\;x\in (\frac1{2^n}G)\cap(-\frac1{2^n}G)\}$ and find a point $x_n\in U_n$ such that $|f(x_n)|\ge\e$. Replacing $x_n$ by $-x_n$, if necessary, we can assume that $f(x_n)\ge\e$ (this follows from the inequality $\frac12(f(x_n)+f(-x_n))\ge f(\theta)=0$ ensured by the mid-point convexity of $f$).

Since $x_n\in\frac1{2^n}G$, for every $k\le n$ the point $2^kx_n$ belongs to $G$ and has prenorm $\|2^kx_n\|\le 2^k\|x_n\|\le\frac{2^{n}}{4^n}=\frac1{2^n}$ in the metric linear space $(X,\|\cdot\|)$. This implies that $(2^nx_n)_{n\in\w}$ is a null-sequence in $X$. Since the set $B$ is not null-finite, there exists $a\in X$ such that the set $\Omega:=\{n\in\IN:a+2^nx_n\in B\}$ is infinite. Then $a\in\overline{B}\subset G$. Choose $k\in\IN$ so large that $2^{-k+1}a\in -G$ (such number $k$ exists as $-G$ is a neighborhood of $\theta$).

Observe that for every number $n\in\Omega$ with $n>k$ the mid-point convexity of $f$ ensures that $$
\begin{aligned}
2^{n-k}\e&\le 2^{n-k}f(x_n)\le f(2^{n-k}x_n)=f\big(\tfrac{-2^{-k+1}a}2+\tfrac{2^{-k+1}a+2^{n-k+1}x_n}2\big)\le\\
&\le\tfrac12f(-2^{-k+1}a)+\tfrac12f(2^{-k+1}a+2^{n-k+1}x_n)\le\tfrac12f(-2^{-k+1}a)+2^{-k}f(a+2^nx_n)
\end{aligned}
$$and hence $$\sup f(B)\ge \sup_{n\in\Omega}f(a+2^nx_n)\ge\sup_{k< n\in\Omega}(2^{n}\e-2^{k-1}f(-2^{-k+1}a))=\infty,$$
which contradicts the upper boundedness of $f$ on $B$.
\end{proof}

Theorem~\ref{t:J} implies the following generalization of the classical Bernstein-Doetsch Theorem \cite{BD}, due to Mehdi \cite{Mehdi}; see also a survey paper of Bingham and Ostaszewski \cite{BO2}.

\begin{corollary} A mid-point-convex function $f:G\to \IR$ defined on an open convex subset $G$ of a metric linear space $X$ is continuous if and only if $f$ is upper-bounded on some non-empty open subset of $G$.
\end{corollary}

Combining Theorem~\ref{t:J} with Theorems~\ref{t:Hm} and \ref{t:Hn}, we obtain the following two continuity criteria, which answer the Problem of Baron and Ger \cite[P239]{BG}.

\begin{corollary}\label{c0} A mid-point convex function $f:G\to\IR$ defined on a convex subset $G$ of a complete metric linear space $X$ is continuous if and only if it is upper bounded on some uBP-set $B\subset G$ which is not Haar-meager in $X$.
\end{corollary}

\begin{proof}
Assume that $f$ is upper bounded on some uBP-set $B\subset G$ which is not Haar-meager in $X$. 

Write the open set $G$ of $X$ as the union $U=\bigcup_{n\in\w}F_n$ of closed subsets of $X$. By \cite{Darji}, the countable union of uBP Haar-meager sets is Haar-meager. Since the set $B=\bigcup_{n\in\w}B\cap F_n$ is not Haar-meager, for some $n\in\IN$ the subset $B\cap F_n$ of $B$ is not Haar-meager. By Theorem~\ref{t:Hm}, $B\cap F_n$ is not null-finite. Since $\overline{B\cap F_n}\subset F_n\subset G$ and $f$ is upper bounded on $B\cap F_n$, we can apply Theorem~\ref{t:J} and conclude that
the mid-point convex function $f$ is continuous.
\end{proof}

\begin{corollary} A mid-point convex function $f:G\to\IR$ defined on a convex subset $G$ of a complete metric linear space $X$ is continuous if and only if it is upper bounded on some universally measurable set $B\subset G$ which is not Haar-null in $X$.
\end{corollary}

\begin{proof} The proof of this corollary runs in exactly the same way as the proof of Corollary~\ref{c0} and uses the well-known fact \cite{Ch} that the countable union of universally measurable Haar-null sets in a complete metric group is Haar-null.
\end{proof}

\section{Some Open Problems}\label{s:prob}

In this section we collect some open problems related to null-finite sets.

It is well-known that for a locally compact Polish group $X$, each Haar-meager set in $X$ can be enlarged to a Haar-meager $F_\sigma$-set and each Haar-null set in $X$ can be enlarged to a Haar-null $G_\delta$-set in $X$. Those ``enlargement'' results dramatically fail for non-locally compact Polish groups, see \cite{EVid}, \cite{DV}. We do not know what happens with null-finite sets in this respect.

\begin{problem} Is each Borel null-finite subset $A$ of a (compact) Polish group $X$ contained in a null-finite set $B\subset X$ of low Borel complexity?
\end{problem}

By \cite{Solecki} (resp. \cite[4.1.6]{EN}), each analytic Haar-null (resp. Haar-meager) set in a Polish group is contained in a Borel Haar-null (resp. Haar-meager) set. On the other hand, each non-locally compact Polish group contains a coanalytic Haar-null (resp. Haar-meager) set which cannot be enlarged to a Borel Haar-null (resp. Haar-meager) set, see \cite{EVid} (resp. \cite{DV}).

\begin{problem} Is each (co)analytic null-finite set $A$ in a Polish group $X$ contained in a Borel null-finite set?
\end{problem}

Our next problem ask about the relation of the $\sigma$-ideal $\sigma\NF_X$ generated by null-finite sets to other known $\sigma$-ideals.

\begin{problem}\label{prob4} Let $X$ be a Polish group.
\begin{enumerate}
\item Is $\sigma\overline{\HN}_X\subset\sigma\NF_X$?
\item Is $\sigma\NF_X=\HN_X\cap\HM_X$?
\end{enumerate}
\end{problem}

The negative answer to both parts of Problem~\ref{prob4} would follow from the negative answer to the following concrete question.

\begin{problem} Can the closed Haar-null set $\prod_{n=2}^\infty C_n^*$ in the group $X=\prod_{n=2}^\infty C_n$ be written as a countable union of Borel null-finite sets in $X$?
\end{problem}


\begin{problem} For an infinite Polish group $X$ and an ideal $\I\in\{\sigma\overline{\NF}_X,\sigma\NF_X\}$ evaluate the standard cardinal characteristics of $\I$:
\begin{itemize}
\item $\add(\I)=\min\{|\J|:\J\subset\I$ and $\bigcup\J\notin\I\}$;
\item $\non(\I)=\min\{|A|:A\subset X$ and $A\notin\I\}$;
\item $\cov(\I)=\min\{|\J|:\J\subset\I$ and $\bigcup\J=X\}$;
\item $\cof(\I)=\min\{|\J|:\forall I\in\I\;\exists J\in\J\;\;(I\subset J)\}$.
\end{itemize}
\end{problem}

The cardinal characteristics of the $\sigma$-ideals $\HN_X$, $\HM_X$, $\sigma\overline{\HN}_X$ on Polish groups are well-studied \cite{BaJu}, \cite{Ban04} and play an important role in Modern Set Theory.


\newpage

\begin{thebibliography}{99}

\bibitem{AT} A.~Arhangel'skii, M.~Tkachenko, {\em  Topological groups and related structures},  Atlantis Studies in Mathematics, 1. Atlantis Press, Paris; World Scientific Publishing Co. Pte. Ltd., Hackensack, NJ, 2008.

\bibitem{Ban04} T.~Banakh, {\em Cardinal characteristics of the ideal of Haar null sets}, Comment. Math. Univ. Carolin. {\bf 45}:1 (2004),  119--137.

\bibitem{BG}
K. Baron, R. Ger, {\em Problem (P239)}, in: \textit{The 21st International Symposium on Functional Equations}, August 6-13, 1983, Konolfingen, Switzerland, Aequationes Math. {\bf 26} (1984) 225--294.

\bibitem{BaJu} T.~Bartoszynski, H.~Judah, {\em Set theory. On the structure of the real line}, A K Peters, Ltd., Wellesley, MA, 1995.

\bibitem{BD} F.~Bernstein, G.~Doetsch, {\em Zur Theorie der konvexen Funktionen}, Math. Ann. {\bf 76}:4 (1915) 514--526.

\bibitem{BO1} N.H.~Bingham, A.J.~Ostaszewski, {\em The Steinhaus-Weil property: its converse, Solecki amenability and subcontinuity}, preprint (https://arxiv.org/abs/1607.00049).

\bibitem{BO2} N.H.~Bingham, A.J.~Ostaszewski, {\em Category-measure duality: convexity, mid-point convexity and Berz sublinearity}, preprint (https://arxiv.org/abs/1607.05750).


\bibitem{Ch}
J.P.R. Christensen, {\em On sets of Haar measure zero in abelian Polish groups}, Israel J. Math. {\bf 13} (1972) 255--260.

\bibitem{Darji}
U.B. Darji, {\em On Haar meager sets}, Topology Appl. {\bf 160} (2013) 2396--2400.

\bibitem{DV} M. Dole\v{z}al, V. Vlas\v{a}k, {\em Haar meager sets, their hulls, and relationship to compact sets}, J. Math.
Anal. Appl. {\bf 446}:1 (2017) 852--863.

\bibitem{DVVR} M. Dole\v{z}al, V. Vlas\v{a}k, B. Vejnar, M. Rmoutil, {\em Haar Meager Sets revisited}, J. Math. Anal. Appl. {\bf 440}:2 (2016) 922--939.

\bibitem{EN} M.~Elekes, D.~Nagy, {\em Haar null and Haar meager sets: a survey and new results}, preprint (https://arxiv.org/abs/1508.02053).

\bibitem{EVid} M. Elekes, Z. Vidny\'anszky, {\em Haar null sets without $G_\delta$ hulls}, Israel J. Math. {\bf 209} (2015) 199--214.

\bibitem{FA} M.~Fabian, P.~Habala, P.~Hajek, V.~Montesinos, J.~Pelant, V.~Zizler, {\em Functional analysis and infinite-dimensional geometry}, Springer-Verlag, New York, 2001.

\bibitem{FS}
P. Fischer, Z. S{\l}odkowski, {\em Christensen zero sets and measurable convex functions}, Proc. Amer. Math. Soc. {\bf 79} (1980) 449--453.

\bibitem{HSY}
B.R.~Hunt, T.~Sauer, J.A.~Yorke, {\em Prevalence: a translation-invariant ``almost every'' on infinite-dimensional spaces}, Bull. Amer. Math. Soc. {\bf 27} (1992) 217--238.

\bibitem{HSY1}
B.R.~Hunt, T.~Sauer, J.A.~Yorke, {\em Prevalence: an addendum}, Bull. Amer. Math. Soc. {\bf 28} (1993) 306--307.

\bibitem{Jab}
E. Jab{\l}o\'{n}ska,  {\em  Some analogies between Haar meager sets and Haar null sets in abelian Polish groups}, J. Math. Anal. Appl. {\bf 421} (2015) 1479--1486.

\bibitem{Ke} A.~Kechris, {\em Classical Descriptive Set Theory}, Springer-Verlag, New York, 1995.

\bibitem{Kuczma} M. Kuczma, {\em An introduction to the theory of functional equations and inequalities. Cauchy's Equation and Jensen's
Inequality}, PWN $\&$ Uniwersytet \'Sl\c aski w Katowicach, Warszawa-Krak\'ow-Katowice, 1985.

\bibitem{Kwela} A.~Kwela, {\em Haar-smallest sets}, (https://arxiv.org/abs/1711.09753).

\bibitem{LP} Ie.~Lutsenko, I.V.~Protasov, {\em Sparse, thin and other subsets of groups}, Internat. J. Algebra Comput. {\bf 19}:4 (2009) 491--510.

\bibitem{MZ} E. Matou\v{s}kov\'{a}, M. Zelen\'{y}, {\em A note on intersections of non--Haar null sets}, Colloq. Math. {\bf 96} (2003) 1--4.

\bibitem{Mehdi} M.R. Mehdi, {\em On convex functions}, J. London Math. Soc. {\bf 39} (1964) 321--326.

\bibitem{Pet} B.J. Pettis, {\em Remarks on a theorem of E. J. McShane}, Proc. Amer. Math. Soc. {\bf 2} (1951) 166--171.

\bibitem{Pic} S. Piccard, {\em Sur les ensembles de distances des ensembles de points d'un espace Euclidien},
Mem. Univ. Neuch\^{a}tel, vol. 13, Secr\'{e}tariat Univ., Neuch\^{a}tel, 1939.

\bibitem{Sch} H.~Schaefer, {\em Topological vector spaces}, Springer--Verlag, New York--Berlin, 1971.

\bibitem{Solecki}
S. Solecki, {\em On Haar null sets}, Fund. Math. {\bf 149} (1996) 205--210.

\bibitem{Steinhaus} H. Steinhaus, {\em Sur les distances des points des ensembles de mesure positive}, Fund. Math. {\bf 1} (1920) 99--104.


\end{thebibliography}
\end{document}